\documentclass[12pt]{article}

\usepackage[left=2cm,right=2cm,top=3cm,bottom=3cm]{geometry}

\usepackage[T1]{fontenc}
\usepackage[utf8]{inputenc}
\usepackage{amsmath}
\usepackage{amssymb}
\usepackage{amsthm}
\usepackage{enumitem}
\usepackage{tikz}
\usetikzlibrary{shapes.geometric,positioning}
\usepackage{graphicx}
\usepackage[colorlinks,urlcolor=blue]{hyperref}
\if10     % 11 = on,  10 = off
\usepackage[mathlines]{lineno}
\newcommand*\patchAmsMathEnvironmentForLineno[1]{%
  \expandafter\let\csname old#1\expandafter\endcsname\csname #1\endcsname
  \expandafter\let\csname oldend#1\expandafter\endcsname\csname end#1\endcsname
  \renewenvironment{#1}%
     {\linenomath\csname old#1\endcsname}%
     {\csname oldend#1\endcsname\endlinenomath}}% 
\newcommand*\patchBothAmsMathEnvironmentsForLineno[1]{%
  \patchAmsMathEnvironmentForLineno{#1}%
  \patchAmsMathEnvironmentForLineno{#1*}}%
\AtBeginDocument{%
\patchBothAmsMathEnvironmentsForLineno{equation}%
\patchBothAmsMathEnvironmentsForLineno{align}%
\patchBothAmsMathEnvironmentsForLineno{flalign}%
\patchBothAmsMathEnvironmentsForLineno{alignat}%
\patchBothAmsMathEnvironmentsForLineno{gather}%
\patchBothAmsMathEnvironmentsForLineno{multline}%
}
\linenumbers
\fi
\newcommand{\N}{\mathbb{N}}

\newcommand{\G}{\mathcal{G}}
\newcommand{\X}{\mathcal{X}}

\numberwithin{equation}{section}

\theoremstyle{plain}
\newtheorem{theorem}{Theorem}%[section]

\newtheorem{claim}[theorem]{Claim}

\newtheorem{corollary}[theorem]{Corollary}

\newtheorem{observation}[theorem]{Observation}

\theoremstyle{definition}

\newtheorem{problem}{Problem}

\theoremstyle{remark}

\usepackage[hang,flushmargin]{footmisc}

\newcommand{\OURURL}{\url{https://arxiv.org/abs/1610.04499}}

\title{Ore and Chv\'{a}tal-type Degree Conditions for Bootstrap Percolation from Small Sets}
\author{Michael Dairyko$^{1,4,5}$, Michael Ferrara$^{2,4,7}$, Bernard Lidick\'y$^{1,4,6}$,\\ Ryan R.~Martin$^{1,4,8}$, Florian Pfender$^{2,4,6,9}$ and Andrew J.~Uzzell$^{3,4}$}
\date{\today}

\begin{document}

\maketitle

\begin{abstract}
Bootstrap percolation is a deterministic cellular automaton in which vertices of a graph~$G$ begin in one of two states, ``dormant'' or ``active''.  Given a fixed positive integer $r$, a dormant vertex becomes active if at any stage it has at least $r$ active neighbors, and it remains active for the duration of the process.  Given an initial set of active vertices $A$, we say that $G$ $r$-percolates (from $A$) if every vertex in $G$ becomes active after some number of steps.  Let $m(G,r)$ denote the minimum size of a set $A$ such that $G$ $r$-percolates from $A$.  

Bootstrap percolation has been studied in a number of settings, and has applications to both statistical physics and discrete epidemiology.  Here, we are concerned with degree-based density conditions that ensure $m(G,2)=2$.  In particular, we give an Ore-type degree sum result that states that if a graph $G$ satisfies $\sigma_2(G)\ge n-2$, then either $m(G,2)=2$ or $G$ is in one of a small number of classes of exceptional graphs.  (Here, $\sigma_2(G)$ is the minimum sum of degrees of two non-adjacent vertices in $G$.) We also give a Chv\'{a}tal-type degree condition:  If $G$ is a graph with degree sequence $d_1\le d_2\le\dots\le d_n$ such that $d_i \geq i+1$ or $d_{n-i} \geq n-i-1$ for all $1 \leq i < \frac{n}{2}$, then $m(G,2)=2$ or $G$ falls into one of several specific exceptional classes of graphs.  Both of these results are inspired by, and extend, an Ore-type result in [D. Freund, M. Poloczek, and D. Reichman, Contagious sets in dense graphs, {\it European J. Combin.}~\textbf{68} 2018].
\end{abstract}

\footnotetext[1]{Department of Mathematics, Iowa State University, Ames, IA.  {\tt \{mdairyko,lidicky,rymartin\}@iastate.edu}.}
\footnotetext[2]{Department of Mathematical and Statistical Sciences, University of Colorado Denver, Denver, CO. {\tt\{michael.ferrara,florian.pfender\}@ucdenver.edu.}}
\footnotetext[3]{Department of Mathematics and Computer Science, College of the Holy Cross. {\tt auzzell@holycross.edu}.}
\footnotetext[4]{Research supported in part by NSF-DMS Grants \#1604458, \#1604773, \#1604697 and \#1603823, ``Collaborative Research: Rocky Mountain - Great Plains Graduate Research Workshops in Combinatorics''.}
\footnotetext[5]{Research supported in part by a PI conference grant from the Institute of Mathematics and its Applications.}
\footnotetext[6]{Research supported in part by NSF grants DMS-1600390 and DMS-1855653 (to B. Lidick\'{y}) and DMS-1600483 and DMS-1855622 (to F. Pfender).}
\footnotetext[7]{Research supported in part by Simons Foundation Collaboration Grants (\#206692 and \#426971, to M. Ferrara).}
\footnotetext[8]{Research supported in part by Simons Foundation Collaboration Grant (\#353292, to R.R. Martin).}
\footnotetext[9]{Research supported in part by Simons Foundation Collaboration Grant (\#276726, to F. Pfender).}

\section{Introduction}

Bootstrap percolation, also known as the {\em irreversible $r$-threshold process} \cite{DreyerRoberts, Roberts} or the \emph{target set selection} is a deterministic cellular automaton first introduced by Chalupa, Leath, and Reich \cite{CLR}. Vertices of a graph are in one of two states, ``dormant'' or ``active.''  Given an integer $r$, a dormant vertex becomes active only if it is adjacent to at least $r$ active vertices. Once a vertex is activated, it remains in that state for the remainder of the process. 

More formally, consider a graph $G$ and let $A$ denote the initial set of active vertices. For a fixed $r \in \N$, the {\em $r$-neighbor bootstrap percolation} process on $G$ occurs recursively by setting $A = A_0$ and for each time step $t\ge 0$,  \[ A_t = A_{t-1} \cup \{ v \in V(G) : | A_{t-1} \cap N(v) | \ge r\},\] where $N(v)$ denotes the neighborhood of the vertex $v$. If all of the vertices of $G$ eventually become active, regardless of order, then we say that $A$ is {\em $r$-contagious} or that $G$ {\em $r$-percolates from $A$}. Given $G$ and $r$, let $m(G,r)$ denote the minimum size of an $r$-contagious set in $G$.   (Observe that $m(G, r) \geq \min\{r, |V(G)|\}$.)
 
Originally, bootstrap percolation was studied on lattices by statistical physicists as a model of ferromagnetism \cite{CLR}, and it can also be viewed as a model of discrete epidemiology, wherein a virus or other contagion is being transmitted across a network (cf.~\cite{BaloghPete,Roberts}).  (In the latter context, each vertex is either ``infected'' or ``uninfected''.)  Further applications include the spread of influence in social networks \cite{Chen, KKT} and market stability in finance \cite{ACM}.  

Much attention has been devoted to examining percolation in a probabilistic setting, referred to in~\cite{BaloghPete} as the \textit{random disease problem}. In this setting, the initial activated set $A$ is selected according to some probability distribution.  The parameter of interest is then the probability that $G$ $r$-percolates from $A$, and in particular determining the threshold probability $p$ for which $G$ almost surely does (or does not) $r$-percolate when vertices are placed in $A$ independently with probability~$p$.  Results have been obtained in this setting for a number of families of graphs, including random regular graphs \cite{BaloghPittel}, the Erd\H{o}s--R\'enyi random graph~$G_{n,p}$ \cite{FKR,JLTV}, expanders~\cite{COFKR15}, hypercubes \cite{BB:hypercube}, trees \cite{BaloghPeresPete}, and grids \cite{AL,BBDCM,BBM:high}.
A degree sequence bound was obtained in~\cite{Reichman12}.

In addition, there has recently been interest in extremal problems concerning bootstrap percolation in various families of graphs~\cite{BenevidesPrzykucki15,Przykucki,Riedl:tree}.

The problem has also been studied from the point of view of computational complexity.  
For $r \geq 3$, determining $m(G, r)$ is NP-complete~\cite{DreyerRoberts}, and determining $m(G, 2)$ is NP-complete even for graphs with maximum degree $4$~\cite{Centeno2011,irrbernard}.  
%Furthermore, it is computationally difficult to approximate $m(G, r)$ to within a polylogarithmic factor unless NP has quasi-polynomial algorithms~\cite{Chen}.
Furthermore, it is computationally difficult to approximate $m(G, r)$~\cite{Chen,charikar_et_al:LIPIcs:2016:6627}.
Notice that $m(G,1)$ is always equal to the number of connected components of~$G$.

\subsection{Degree-Based Results}

In this paper, we are interested in degree-based density conditions that ensure that a graph $G$ will percolate from a small set of initially activated vertices.  
Let $\delta(G)$ denote the minimum degree of $G$.
Freund, Poloczek, and Reichman~\cite{FPR} showed that for each $r \geq 2$, if $G$ has order~$n$ and $\delta(G) \geq \frac{r-1}{r}n$, then $m(G, r) = r$.  Note that when $r = 2$, this is the same as Dirac's condition for hamiltonicity~\cite{D}.  Recently, Gunderson~\cite{Gunderson} showed that if $n \geq 30$ and $\delta(G) \geq \lfloor n/2 \rfloor + 1$, then $m(G, 3) = 3$, and that for each $r \geq 4$, if $n$ is sufficiently large and $\delta(G) \geq \lfloor n/2 \rfloor + r - 3$, then $m(G, r) = r$.  Moreover, both bounds are sharp.

Let $\sigma_2(G)$ denote the minimum degree sum of a pair of nonadjacent vertices in a graph $G$, i.e.
\[
\sigma_2(G) = \min\{d(x) + d(y) : xy \notin E(G)\}.
\]
  Ore~\cite{O} proved that every graph $G$ of order~$n\ge 3$ that satisfies $\sigma_2(G)\ge n$ is hamiltonian.  Freund, Poloczek, and Reichman \cite{FPR} also showed that Ore's condition is sufficient to ensure that a graph 2-percolates from the smallest possible initially activated set.

\begin{theorem}[\cite{FPR}]\label{th:FPROre}
Let $n \geq 2$.  If $G$ is a graph of order~$n$ and $\sigma_2(G) \geq n$, then $m(G, 2) = 2$.
\end{theorem}

Recently, Wesolek~\cite{Wesolek} proved a bound on $\sigma_2(G)$ that implies that $m(G, r) = r$ for all $r \geq 3$.

%As noted in \cite{FPR}, Theorem \ref{th:FPROre} immediately implies an analogue to Dirac's Theorem \cite{D}, in that a graph $G$ of order $n$ with minimum degree at least $\frac{n}{2}$ must also satisfy $m(G,2)=2$.
Note that hamiltonicity alone is not sufficient to conclude that a graph~$G$ satisfies $m(G, 2) = 2$, as $m(C_n,2) = \lceil n/2 \rceil$, which tends to infinity with $n$.  Rather, Theorem \ref{th:FPROre} is part of a diverse collection of results that demonstrate that many sufficient density conditions for hamiltonicity imply a much richer structure that allows for stronger conclusions (cf.~\cite{Bondy:metaconjecture,BCFGL}).  

In this paper, we improve Theorem~\ref{th:FPROre} in several ways.  First, we characterize graphs of order $n$ with $\sigma_2\ge n-2$ and $m(G,2)>2$.  
These will consist of four infinite families of graphs $\G_0$, $\G_1$, $\G_2$, $\G_3$ and a finite set of graphs~$\X$.
The graphs in $\X$ are depicted in Figure \ref{fig:extremal}.

The class $\G_0$  consists of all graphs which are unions of two disjoint non-empty cliques $X$,~$Y$.
Note that $X$ and $Y$ can be of different sizes.
Graphs in $\G_1$, $\G_2$ and $\G_3$ are formed from $\G_0$ by selecting $\{x,x'\}\subseteq X$, $\{y,y'\}\subseteq Y$, adding the edges $xy$, $x'y'$, and deleting the edges $xx'$ and $yy'$ if they exist (see Figure~\ref{fig:Gi}).  For simplicity, we have distinguished the cases where $x=x'$ and $y=y'$ ($\G_1$), $x\ne x'$ and $y\ne y'$ ($\G_2$) and  $x=x'$ and $y\ne y'$ ($\G_3$).  It is easy to see that any graph $G\in \G_0\cup\G_1\cup\G_2\cup\G_3$ containing at least one vertex in each of $X$ and $Y$ that is not adjacent to any vertex in the other set has $\sigma_2(G)=|V(G)|-2$ and $m(G,2)>2$.

\begin{figure}
\begin{center}
\[ \begin{array}{cccccccc}
\begin{tikzpicture}[scale=1]
\foreach \a in {5}{
\node[regular polygon, regular polygon sides=\a, minimum size=2cm, draw] at (\a*4,0) (A) {};
\foreach \i in {1,...,\a}
    \node[circle,fill=black,scale=.5] at (A.corner \i) {};
}
\end{tikzpicture}

&

\begin{tikzpicture}[scale=1]
\foreach \a in {5}{
\node[regular polygon, regular polygon sides=\a, minimum size=2cm, draw] at (0,0) (A) {};
\foreach \i in {1,...,\a}
    \node[circle,fill=black,scale=.5](y\i) at (A.corner \i) {};
}
  \node[circle,fill=black,scale=.5] (a) at (0, 0) {}; 

  \draw (y3) -- (a) -- (y4);
 
\end{tikzpicture}
&

\begin{tikzpicture}[scale=1.2]
  \node[circle,fill=black,scale=.5] (a) at (0,0){}; 
  \node[circle,fill=black,scale=.5] (b) at (1,0) {}; 
  \node[circle,fill=black,scale=.5] (c) at (1,.75) {}; 
  \node[circle,fill=black,scale=.5] (d) at (1, 1.5) {}; 
  \node[circle,fill=black,scale=.5] (e) at (0 , 1.5) {}; 
  \node[circle,fill=black,scale=.5] (f) at (0,.75) {}; 
  \draw (a) -- (b);
  \draw (b) -- (c);
  \draw(c) -- (d);
  \draw(d) -- (f);
  \draw(e) -- (c);  
  \draw(e) -- (f);  
  \draw(f) -- (a);  
\end{tikzpicture}

&

\begin{tikzpicture}[scale=1.2]
  \node[circle,fill=black,scale=.5] (a) at (0,0){}; 
  \node[circle,fill=black,scale=.5] (b) at (1,0) {}; 
  \node[circle,fill=black,scale=.5] (c) at (1,.75) {}; 
  \node[circle,fill=black,scale=.5] (d) at (1, 1.5) {}; 
  \node[circle,fill=black,scale=.5] (e) at (0 , 1.5) {}; 
  \node[circle,fill=black,scale=.5] (f) at (0,.75) {}; 
 
  \draw (a) -- (b);
  \draw (b) -- (c);
  \draw(c) -- (d);
  \draw(d) -- (e);
  \draw(e) -- (f);  
  \draw(f) -- (c);  
  \draw(f) -- (a);  
\end{tikzpicture}

\\[10pt]

\begin{tikzpicture}[xscale=0.8,yscale=0.9]
\foreach \a in {0,1,2}{    
    \node[circle,fill=black,scale=.5](z\a) at (1,-1+\a) {};
    \node[circle,fill=black,scale=.5](x\a) at (0,-1+\a) {};    
\draw (z\a) -- (x\a);
}
\foreach \a in {0,1}{    
    \node[circle,fill=black,scale=.5](y\a) at (-1,-0.5+\a) {};
\draw (y\a) -- (x0) (y\a) -- (x1) (y\a) -- (x2);
}
\draw (z0)--(z1)--(z2) (z0) to[bend right=50] (z2);

%\draw (x)--(zx)--(zy)--(y)--(z0)--(x)--(z3)--(y) (z0)--(z1)--(z2) (z1)--(z3) (zx)--(z2)--(zy);
\end{tikzpicture}

%\hskip 2em
%
%\begin{tikzpicture}[xscale=0.6,yscale=0.7]
%    \node[circle,fill=black,scale=.5](x) at (-2,0) {};
%    \node[circle,fill=black,scale=.5](y) at (2,0) {};
%\foreach \a in {0,2,3}{    
%    \node[circle,fill=black,scale=.5](z\a) at (0,-1.5+\a) {};
%}
%\node[circle,fill=black,scale=.5](z1) at (-0.5,-0.5) {};
%    \node[circle,fill=black,scale=.5](zy) at (1,0) {};
%    \node[circle,fill=black,scale=.5](zx) at (-1,0) {};
%\draw (x)--(zx)--(zy)--(y)--(z0)--(x)--(z3)--(y) (z0)--(z1)--(z2) (z1)--(z3) (zx)--(z2)--(zy);
%\end{tikzpicture}

&

\begin{tikzpicture}[scale=.9]
\foreach \a in {5}{
\node[regular polygon, regular polygon sides=\a, minimum size=2cm, draw] at (0,0) (A) {};
\foreach \i in {1,...,\a}
    \node[circle,fill=black,scale=.5](y\i) at (A.corner \i) {};
}
\node[circle,fill=black,scale=.5] (a) at (270+72:0.4) {}; 
\node[circle,fill=black,scale=.5] (b) at (270-72:0.4) {}; 
\node[circle,fill=black,scale=.5] (c) at (90:0.4) {}; 
\draw
(a)--(c)--(b)
(y1)--(c)
(y2)--(b)--(y3)
(y4)--(a)--(y5)
;
\end{tikzpicture}

& 

\begin{tikzpicture}[scale=.8]
\foreach \a in {4}{
\node[regular polygon, regular polygon sides=\a, minimum size=2cm, draw] at (0,0) (A) {};
\foreach \i in {1,...,\a}
    \node[circle,fill=black,scale=.5](x\i) at (A.corner \i) {};
}
\foreach \a in {4}{
\node[regular polygon, regular polygon sides=\a, minimum size=2cm, draw] at (.75,.75) (A) {};
\foreach \i in {1,...,\a}
    \node[circle,fill=black,scale=.5](y\i) at (A.corner \i) {};
}
\draw
(x1)--(y2)
(x2)--(y1)
(x3)--(y3)
(x4)--(y4)
;
\end{tikzpicture}

&

\begin{tikzpicture}[scale=.8]
\foreach \a in {4}{
\node[regular polygon, regular polygon sides=\a, minimum size=2cm, draw] at (0,0) (A) {};
\foreach \i in {1,...,\a}
    \node[circle,fill=black,scale=.5](x\i) at (A.corner \i) {};
}
\foreach \a in {4}{
\node[regular polygon, regular polygon sides=\a, minimum size=2cm, draw] at (.75,.75) (A) {};
\foreach \i in {1,...,\a}
    \node[circle,fill=black,scale=.5](y\i) at (A.corner \i) {};
}
\draw
(x1)--(y1)
(x2)--(y2)
(x3)--(y3)
(x4)--(y4)
;
\end{tikzpicture}

\end{array}\]
\caption{The family $\X$: Small exceptional graphs for Theorem \ref{th:bestOre}.}
\label{fig:extremal}
\end{center}
\end{figure}
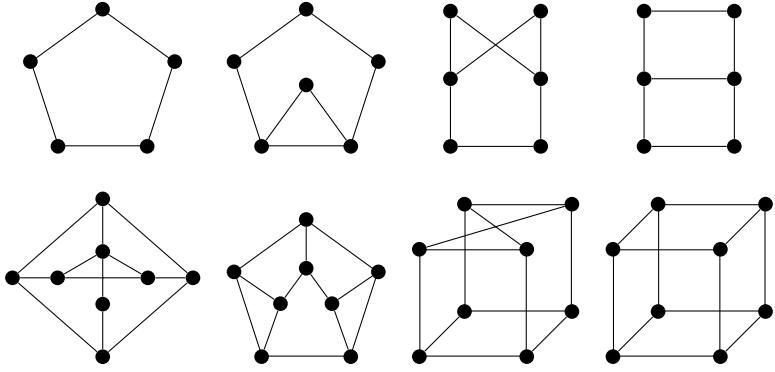

%%%%%%%%%%%%%%%%%%%%%%%%%%%%%%%%%%%%%%%%%%%%%%%%%%%%%%%%%%%%%%

\begin{figure}
\begin{center}
\tikzset{vtx/.style={inner sep=1.7pt, outer sep=0pt, circle,fill},}
\[ \begin{array}{cccccccc}

\begin{tikzpicture}[scale=.9]
\foreach \a in {4}{
\node[regular polygon, regular polygon sides=\a, minimum size=2cm, draw] at (0,0) (A) {};
\foreach \i in {1,...,\a}
    \node[circle,fill=black,scale=.5](x\i) at (A.corner \i) {};
}
\draw (x1)--(x3)--(x2)--(x4)--(x1);

\foreach \a in {6}{
\node[regular polygon, regular polygon sides=\a, minimum size=2cm, draw] at (3,0) (A) {};
\foreach \i in {1,...,\a}
    \node[circle,fill=black,scale=.5](y\i) at (A.corner \i) {};
}
\draw (y1)--(y3)--(y5)--(y2)--(y4)--(y1) (y1)--(y6)--(y2) (y3)--(y6)--(y4) (y5)--(y1);

\end{tikzpicture}
&

\begin{tikzpicture}[scale=.9]
\foreach \a in {5}{
\node[regular polygon, regular polygon sides=\a, minimum size=2cm, draw] at (0,0) (A) {};
\foreach \i in {1,...,\a}
    \node[circle,fill=black,scale=.5](x\i) at (A.corner \i) {};
}
\draw (x1)--(x3)--(x5)--(x2)--(x4)--(x1);

\foreach \a in {5}{
\node[regular polygon, regular polygon sides=\a, minimum size=2cm, draw] at (3,0) (A) {};
\foreach \i in {1,...,\a}
    \node[circle,fill=black,scale=.5](y\i) at (A.corner \i) {};
}
\draw (y1)--(y3)--(y5)--(y2)--(y4)--(y1);

\draw (x5) -- (y2);

\end{tikzpicture}  \\

\G_0  & \G_1 \\

\begin{tikzpicture}[scale=.9]
\foreach \a in {5}{
\node[regular polygon, regular polygon sides=\a, minimum size=2cm] at (0,0) (A) {};
\foreach \i in {1,...,\a}
    \node[circle,fill=black,scale=.5](x\i) at (A.corner \i) {};
}
\draw (x1)--(x3)--(x5)--(x2)--(x4)--(x1);
\draw (x1)--(x2)--(x3)--(x4)  (x5)--(x1);

\foreach \a in {5}{
\node[regular polygon, regular polygon sides=\a, minimum size=2cm] at (3,0) (A) {};
\foreach \i in {1,...,\a}
    \node[circle,fill=black,scale=.5](y\i) at (A.corner \i) {};
}
\draw (y1)--(y3)--(y5)--(y2)--(y4)--(y1);
\draw (y1)--(y2)  (y3)--(y4)--(y5)--(y1);

\draw (x5) -- (y2) (x4)--(y3);
\begin{scope}[node distance=3pt]
\node [above=of x5] {$x$};
\node [below=of y3] {$v$};
\end{scope}
\end{tikzpicture}
&

\begin{tikzpicture}[scale=.9]
\foreach \a in {5}{
\node[regular polygon, regular polygon sides=\a, minimum size=2cm, draw] at (0,0) (A) {};
\foreach \i in {1,...,\a}
    \node[circle,fill=black,scale=.5](x\i) at (A.corner \i) {};
}
\draw (x1)--(x3)--(x5)--(x2)--(x4)--(x1);

\foreach \a in {5}{
\node[regular polygon, regular polygon sides=\a, minimum size=2cm] at (3,0) (A) {};
\foreach \i in {1,...,\a}
    \node[circle,fill=black,scale=.5](y\i) at (A.corner \i) {};
}
\draw (y1)--(y3)--(y5)--(y2)--(y4)--(y1);
\draw (y1)--(y2)  (y3)--(y4)--(y5)--(y1);

\draw (x5) -- (y2) (x5)--(y3);
\begin{scope}[node distance=3pt]
\node [above=of x5] {$x$};
\node [below=of y3] {$v$};
\end{scope}

\end{tikzpicture} 
\\
\G_2 & \G_3

\end{array}\]
\end{center}

\caption{Examples of graphs in $\G_0$, $\G_1$, $\G_2$, and $\G_3$ for $n=10$.  The labeled vertices in the third and fourth graphs refer to the proof of Theorem~\ref{th:bestOre}.}

 \label{fig:Gi}
\end{figure}
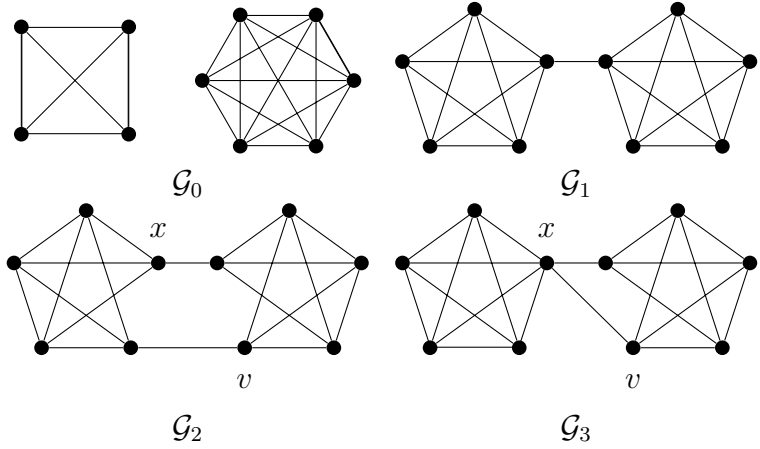

\begin{theorem}\label{th:bestOre}
Let $G$ be a graph of order~$n \geq 2$ such that $G$ is not in $\G_0$, $\G_1$, $\G_2$, $\G_3$ or $\X$.  If $\sigma_2(G) \geq n - 2$, then $m(G, 2) = 2$.
\end{theorem}

In particular, Theorem \ref{th:bestOre} implies that $C_5$ is the only graph~$G$ with $\sigma_2(G) = |V(G)| - 1$ and $m(G, 2) > 2$.

Second, we prove a degree sequence condition for $m(G,2) = 2$.  
Let $G$ be a graph with degree sequence $d_1 \leq \cdots \leq d_n$. 
We say that $G$ satisfies \emph{Chv\'atal's condition} if  
\begin{align}
d_i \geq i+1 \quad\text{ or }\quad d_{n-i} \geq n-i, \quad \forall i, 1 \leq i < \tfrac{n}{2}.\label{eqch}
\end{align}   
In \cite{Chv:hamiltonian}, Chv\'{a}tal proved that a graph $G$ of order $n\ge 3$ that satisfies Chv\'{a}tal's condition is hamiltonian. 

Here, we show that, with only a few exceptions, a slightly weaker Chv\'{a}tal-type condition implies that $m(G,2)=2$.  We say that a graph $G$ satisfies the \emph{weak Chv\'{a}tal condition} if 
\begin{align}
d_i \geq i+1\quad\text{ or }\quad d_{n-i} \geq n-i-1,  \quad \forall i, 1 \leq i < \tfrac{n}{2},\label{eqwch}
\end{align} 
and prove the following.  We denote the path on $k$ vertices by $P_k$ and the cycle on $k$ vertices by $C_k$.

\begin{theorem}\label{thm:chvatalsharp}
If $G$ is a graph with degree sequence $d_1 \leq \cdots \leq d_n$ that satisfies the weak Chv\'{a}tal condition \eqref{eqwch}, then either $m(G,2)=2$ or one of the following holds: 
\begin{itemize}
\item $G$ is disconnected,
\item $G$ contains exactly two vertices of degree one and $G \not\in\{P_2,P_3\}$, or
\item $G$ is $C_5$.
\end{itemize}
\end{theorem}

Note that the ordinary Chv\'atal condition~\eqref{eqch} rules out the last three cases in Theorem~\ref{thm:chvatalsharp}.

\begin{corollary}\label{cor:chvatalcorollary}
If $G$ is a graph with degree sequence $d_1 \leq \cdots \leq d_n$ that satisfies Chv\'{a}tal's condition~\eqref{eqch}, then $m(G,2)=2$.
\end{corollary}

Much as Chv\'{a}tal's Theorem implies Ore's Theorem for hamiltonicity, each of Theorems \ref{th:bestOre} and~\ref{thm:chvatalsharp} and Corollary~\ref{cor:chvatalcorollary} implies Theorem~\ref{th:FPROre}.

\subsection{Notation}
Let $G$ be a graph. 
By $V(G)$ we denote the set of vertices of $G$.
Let $U \subseteq V(G)$ and let $v \in V(G)$.
We denote by $G[U]$ the subgraph of $G$ induced by $U$.
The notations $\Delta(G)$ and $\delta(G)$ mean the maximum and minimum degree of $G$, respectively. By $N(v)$ we denote the set of neighbors of~$v$.
We denote by $N_U(v)$ the set of neighbors in $U$, that is, $N(v) \cap U$.
The notation $d(v)$ means the degree of $v$
 and $d_U(v)$ is $|N_U(v)|$.
%We use $d(v)$ to denote the degree of $v$ and $d_U(v)$ to denote $|N_U(v)|$.
%We say that $N_I(v)$ is the neighborhood of a vertex $v$ within the set $I$. $D_U(u)$ is the degree of the vertex $u$ in $U$. Also, the induced subgraph of $U$ is $G[U]$. 
%Let $S \subseteq V(G)$ and $v \in S$.  

\section{Proof of Theorem~\ref{th:bestOre}}

\begin{proof}[Proof of Theorem~\ref{th:bestOre}]
Let $G$ be a graph of order $n$ with $\sigma_2(G)\ge n-2$ that is not in one of the exceptional classes $\X$, $\G_0$, $\G_1$, $\G_2$, or $\G_3$.  Throughout the proof, amongst all subsets of $V(G)$ that can be activated from a starting set of two vertices, let $I$ (for ``infected'') have maximum size,  and let $U = V(G) \setminus I$ denote the set of vertices that remain dormant from this starting set.   
We repeatedly use the following observation that follows from the maximality of $I$.
\begin{observation}\label{oneneighbor}
Each vertex in $U$ has at most~one neighbor in $I$.
\end{observation}
Notice that our assumption on $\sigma_2(G)$ implies that $\Delta(G) \geq (n - 2)/2$.  

For $n\le 11$, Nauty \cite{McKay2014} was utilized to generate all graphs with $m(G,2)>2$, which is precisely the set $\X$. 
The program we used is available at \OURURL.
Thus, we may assume that $n\ge 12$ as we proceed.  Further, if $G$ is disconnected, the degree sum condition guarantees that $G$ has exactly two complete components and so $G \in \G_0$, a contradiction. We may therefore assume that $G$ is connected. The following sequence of claims establishes important facts about the size and structure of $I$ and $U$.  

\begin{claim}\label{claim:bigI}
$|I|>\frac{n}{2}$ and $G[U]$ is complete.
\end{claim}

\begin{proof} 
First we show that $|I| \geq 4$.
Suppose for a contradiction that $|I| < 4$.  If $G$ contains a copy~$T$ of~$K_3$, then, since $G$ is connected, some vertex $y\in V(T)$ has a neighbor $y'$ in $G-T$.  So, if we initially activate $y'$ and some $x\in V(T) \setminus \{y\}$, then at least 4 vertices are activated.

If $G$ contains a copy of~$C_4$, we can activate the entire cycle starting with either pair of nonadjacent vertices. Hence we may suppose that $G$ contains neither a triangle nor a copy of~$C_4$.  

Let $w$ be a vertex with $d(w)=\Delta(G)\ge (n-2)/2\ge 5$. As $G$ is triangle-free, $N(w)$ is independent. Let $x$ and $y$
%and $z$
be distinct neighbors of $w$.
%and suppose that $d(x)$,~$d(y)\le d(z)$.
As $G$ contains no $C_4$, we have $N(x) \cap N(y) = \{w\}$, which means that $d(x)+d(y) \leq n - 3$,
%$N(x)$, $N(y)$ and $N(z)$ intersect pairwise only in $w$, so that $d(x)+d(y)+d(z) \le n-1$.  This, however, is
a contradiction.  So, we may assume that $|I|\ge 4$.

Next we establish that $|I|\ge |U|$.
It suffices to show that $m(G[U],2)=2$, which would imply that $|U|\le |I|$ since $|I|$ is maximum among all sets activated by two vertices.
To that end, let $u$ and $v$ be nonadjacent vertices in $U$ and recall that every vertex in $U$ has at most~one neighbor in $I$.  Consequently, as $|I|\ge 4$, 
\[
d_U(u)+d_U(v) \ge d(u)-1 + d(v)-1 \ge n-4\ge |U|.
\]
Thus, $G[U]$ satisfies Ore's condition, and so $m(G[U],2) = 2$ by Theorem~\ref{th:FPROre}.  Therefore, $|U|\le |I|$, which implies that the first conclusion of the claim, that $|I| > \frac{n}{2}$, holds unless $|U|=|I|=\frac{n}{2}$. 
 We will deal with the case $|U|=|I|=\frac{n}{2}$ after establishing that $G[U]$ is complete.
 We continue with the assumption that $|I| \geq \frac{n}{2}$.
 
Suppose $G[U]$ is not a complete graph, and let $u$ and $v$ be nonadjacent vertices in $U$. 
Recall that by Observation~\ref{oneneighbor} every vertex in $U$ has at most one neighbor in $I$. Hence $\max(d(u),d(v)) \leq |U|-1$.
Then, as $|U|\le\frac{n}{2}$,
\[
d(u)+d(v)\le 2(|U|-1) \le n-2,
\]
which is a contradiction unless equality holds. If equality holds, then $U$ induces a complete graph on exactly $\frac{n}2$ vertices minus a matching $M$, where every vertex of~$M$ has a neighbor in $I$. Notice that in this case, $G[U]$ percolates from any choice of two vertices in $U$ (since $\frac{n}{2} \geq 6$). Activating two neighbors of a vertex of~$M$ -- one in $I$ and one in $U$ -- results in at least $|U|+1$ activated vertices, a contradiction to $|I|=\frac{n}{2}$ being maximum. Therefore, $G[U]$ must be a complete graph.

We finally return to the case where  $|U|=|I|=\frac{n}{2}$.
Let $v \in I$ have a neighbor $u$ in $U$.  For any $z$ in $U\setminus \{u\}$, 
initially activating $\{v,z\}$ leads to (at least) the activation of $U \cup \{v\}$, contradicting the maximality of $|I|$ and establishing Claim~\ref{claim:bigI}.
\end{proof}

Partition $I$ into sets $I_0$ and $I_1$, where $I_1$ is the set of vertices of~$I$ with at least~one neighbor in $U$, so that vertices in $I_0$ have no neighbors in  $U$.
Since $|I|>|U|$, and no vertex in $U$ has more than one neighbor in $I$, there exists a vertex $w \in I_0$. Let $u\in U$ and observe that
\begin{equation}\label{I0degreebound}
n-2\le d(w)+d(u)\le |I|-1+|U|=n-1.
\end{equation}
%Hence $w$ has at most one non-neighbor in $I$

The bound on $d(w)$ in~\eqref{I0degreebound} has the following useful consequences.

\begin{observation}\label{I0}
Each vertex in $I_0$ has at most~one non-neighbor in $I$.  Furthermore, if any three vertices in $I$ are activated, then all of $I_0$ will be activated in the following step of the percolation.
\end{observation}
%so that $G[I_0]$ is a complete graph minus a matching.

\begin{claim}\label{allinU}
Every vertex in $U$ has exactly one neighbor in $I$. 
\end{claim}

\begin{proof}
By Observation~\ref{oneneighbor}, it suffices to show that each vertex in $U$ has at least one neighbor in $I$.
Suppose otherwise, so that there exists $z \in U$ with no neighbors in $I$.  Thus, by Observation~\ref{oneneighbor},  there are at least two~vertices $w_1$ and $w_2$ in $I_0$.
Since $w_i$ and $z$ are nonadjacent, we get \[n-2 \leq \delta_2(G) \leq  d(w_i) + d(z) \leq |I|-1+|U|-1 = n-2\] for $i \in \{1,2\}$. 
Hence $d(w_1)=d(w_2)=|I|-1$
and $w_1,w_2$ are adjacent to all vertices of $I$.
Let $v \in I$ and $u \in U$ be adjacent vertices. If we initially activate $\{w_1,u\}$, this in turn would activate at least~$I\cup\{u\}$, contradicting the maximality of $|I|$.  Consequently, every vertex in $U$ has a neighbor in $I$, establishing Claim~\ref{allinU}.
\end{proof}
Note that Claim~\ref{allinU} implies that $|U| \geq |I_1|$.

\begin{claim}\label{bigI0}
$|I_0| \geq 2$.
\end{claim}

\begin{proof}
As observed just after the proof of Claim~\ref{claim:bigI}, $I_0$ is non-empty, so suppose for a contradiction that $|I_0| = 1$.  It follows from Claims \ref{claim:bigI} and~\ref{allinU} and $I = I_0 \cup I_1$ that
\[
|I| > |U| \geq |I_1| = |I| - 1,
\]
which is a contradiction unless $|U| = |I| - 1$.

Because $|U| = |I_1|$, Claim~\ref{allinU} implies that there is a perfect matching between $U$ and $I_1$.  Also, $I_1$ cannot be an independent set, or else for all $a$,~$b \in I_1$, 
\[
d(a) + d(b) \leq 2(|I_1| + 1) = 4 < n - 2,
\] a contradiction.  So, let $x_1$ and $x_2$ be adjacent vertices of~$I_1$ and let $u_1$ and $u_2$ be their respective neighbors in $U$.  If we initially activate $\{u_1, x_2\}$, then $x_1$ will also become active.  Furthermore, by Claim~\ref{claim:bigI}, $U$ is a clique, so all of $U$ will become active, for a total of at least~$|U| + 2 =|I| + 1$ active vertices.  This contradiction completes the proof of Claim~\ref{bigI0}.
\end{proof}

\begin{claim}\label{cl:three}

Let $v \in I_1$ have at least two neighbors in $U$ and let $u$ be one such neighbor. Also, let $D$ be a subset of $I$ containing at least three vertices, including $v$, and let $x\in I_1\setminus \{v\}$.  The following hold:

\begin{itemize}
\item[(1)] There is no set of size 2 that activates $U \cup D$;
\item[(2)] $N_I(v)$ is an independent set;
\item[(3)] If there is a vertex $y$ in $N_I(v) \cap I_1$, then $y$ is the only neighbor of $v$ in $I$;
\item[(4)] $v$ and $x$ have no common neighbor;
\item[(5)] $|N_I(v)|=1$,  $I_1 = \{v,x\}$, $x$ has exactly one neighbor in $U$ and $v$ is adjacent to every other vertex of~$U$.
\end{itemize}

\end{claim}

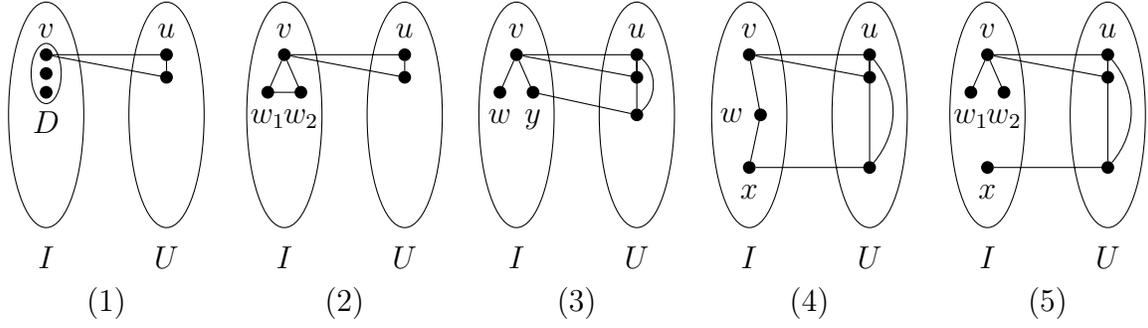
\begin{figure}
\begin{center}
\tikzset{vtx/.style={inner sep=1.7pt, outer sep=0pt, circle,fill},}
\def\xpos{1.6}
\def\hs{\hskip 0.8em}
\begin{tikzpicture}
\draw
(0,0) ellipse (0.5 and 1.5)
(0,0.8) node[vtx,label=above:$v$](v){}
(\xpos,0) ellipse (0.5 and 1.5)
(\xpos,0.8) node[vtx,label=above:$u$](u){}--(v)
(\xpos,0.5) node[vtx,label=above:$ $](up){}--(v)
(u)--(up)
%(0,-0.7) node[vtx,label=below:$x$](x){}
%(2,-0.7) node[vtx,label=above:$ $]{}--(x)
(\xpos,-1.9) node{$U$}
(0,-1.9) node{$I$}
(0,0.3) node[vtx]{}
(0,0.55) node[vtx]{} ellipse(0.2 and 0.4)
(0,-0.1) node{$D$}
(0.5*\xpos,-2.5) node{(1)};
;
\end{tikzpicture}
\hs
\begin{tikzpicture}
\draw
(0,0) ellipse (0.5 and 1.5)
(0,0.8) node[vtx,label=above:$v$](v){}
(\xpos,0) ellipse (0.5 and 1.5)
(\xpos,0.8) node[vtx,label=above:$u$](u){}--(v)
(\xpos,0.5) node[vtx,label=above:$ $](up){}--(v)
(u)--(up)
%(0,-0.7) node[vtx,label=below:$x$](x){}
%(2,-0.7) node[vtx,label=above:$ $]{}--(x)
(\xpos,-1.9) node{$U$}
(0,-1.9) node{$I$}
(v)--(-0.22,0.3) node[vtx,label=below:$w_1$]{}--
(0.22,0.3) node[vtx,label=below:$w_2$]{}--(v)
;
\draw (0.5*\xpos,-2.5) node{(2)};
\end{tikzpicture}
\hs
\begin{tikzpicture}
\draw
(0,0) ellipse (0.5 and 1.5)
(0,0.8) node[vtx,label=above:$v$](v){}
(\xpos,0) ellipse (0.5 and 1.5)
(\xpos,0.8) node[vtx,label=above:$u$](u){}--(v)
(\xpos,0.5) node[vtx,label=above:$ $](up){}--(v)
(u)--(up)
%(0,-0.7) node[vtx,label=below:$x$](x){}
%(2,-0.7) node[vtx,label=above:$ $]{}--(x)
(\xpos,-1.9) node{$U$}
(0,-1.9) node{$I$}
(v)--(-0.22,0.3) node[vtx,label=below:$w$]{}
(0.22,0.3) node[vtx,label=below:$y$](y){}--(v)
(y) -- (\xpos,0) node[vtx](z){}
(u)--(z) (u) to[bend left=50] (z)
;
\draw (0.5*\xpos,-2.5) node{(3)};
\end{tikzpicture}
\hs
\begin{tikzpicture}
\draw
(0,0) ellipse (0.5 and 1.5)
(0,0.8) node[vtx,label=above:$v$](v){}
(\xpos,0) ellipse (0.5 and 1.5)
(\xpos,0.8) node[vtx,label=above:$u$](u){}--(v)
(\xpos,0.5) node[vtx,label=above:$ $](up){}--(v)
(u)--(up)
(0,-0.7) node[vtx,label=below:$x$](x){}
(\xpos,-0.7) node[vtx,label=above:$ $](z){}--(x)
(\xpos,-1.9) node{$U$}
(0,-1.9) node{$I$}
(v)--(0.15,0) node[vtx,label=left:$w$]{}--(x)
(u)--(z) (u) to[bend left=40] (z)
;
\draw (0.5*\xpos,-2.5) node{(4)};
\end{tikzpicture}
\hs
\begin{tikzpicture}
\draw
(0,0) ellipse (0.5 and 1.5)
(0,0.8) node[vtx,label=above:$v$](v){}
(\xpos,0) ellipse (0.5 and 1.5)
(\xpos,0.8) node[vtx,label=above:$u$](u){}--(v)
(\xpos,0.5) node[vtx,label=above:$ $](up){}--(v)
(u)--(up)
(0,-0.7) node[vtx,label=below:$x$](x){}
(\xpos,-0.7) node[vtx,label=above:$ $](z){}--(x)
(\xpos,-1.9) node{$U$}
(0,-1.9) node{$I$}
(v)--(-0.22,0.3) node[vtx,label=below:$w_1$]{} 
(0.22,0.3) node[vtx,label=below:$w_2$]{}--(v)
(u)--(z) (u) to[bend left=40] (z)
;
\draw (0.5*\xpos,-2.5) node{(5)};
\end{tikzpicture}

\end{center}
\caption{The situation in Claim~\ref{cl:three}.}\label{figlemma}
\end{figure}

\begin{proof} 
Before we begin, it is useful to note that if $u$ and $v$ are activated, then so too will be all of~$U$, as $G[U]$ is complete.  Also, the proofs of {\it(1)--(5)} are illustrated in Figure~\ref{figlemma}.\\

\noindent{\it (1):} Suppose that  $U \cup D$ can be activated starting from two vertices $a$ and $b$.
By Observation~\ref{I0}, all of~$I_0$ is activated.
%Let $w\in I_0$.  By Claim~\ref{claim:bigI2orexceptional} and applying the bound on $\sigma_2(G)$ to $w$ and $u$,
%the vertex $w$ can have at most one non-neighbor in $G[I]$.
%%If $w \in I_0$, then by Claim~\ref{claim:bigI2orexceptional} and our assumption on~$\sigma_2(G)$, $w$ has at most~one non-neighbor in $G[I]$.  
%Hence, $w$ has at~least two neighbors in $D$ and so becomes activated as well.
Thus, the set of vertices activated starting with $\{a,b\}$ contains $U \cup I_0 \cup \{v\}$.  As observed above, Claim~\ref{allinU} implies that $|U| \geq |I_1|$.  So, $|U \cup I_0 \cup \{v\}| > |I_0 \cup I_1| = |I|$, a contradiction.\\

\noindent{\it (2):}  Suppose otherwise, and let $w_1$ and $w_2$ be adjacent vertices in $N_I(v)$. Initially activate $\{w_1,u\}$. After the first step, $\{w_1,u,v\}$ are activated. Then all of $\{u,v,w_1,w_2\}\cup U$ is activated, contradicting {\it(1)} with $D=\{v,w_1,w_2\}$.\\   

\noindent{\it (3):}  Assume otherwise, that $v$ has two neighbors $y$ and $w$ in $I$, where $y$ has a neighbor in $U$.  Initially activating $\{w,u\}$ activates $v$ in the first step. Then the other neighbor of $v$ in $U$ and the rest of $U$ are activated in the two steps that follow.  Finally, $y$ is activated, contradicting {\it(1)} with $D=\{v,w,y\}$.\\ 

\noindent{\it (4):}  Let $x$ be in $I_1\setminus \{v\}$, as given, and assume that $w$ is a common neighbor of $v$~and~$x$.  Initially activating $\{w,u\}$ then activates $v$, $x$ and the entirety of $U$ in four iterations, again contradicting {\it(1)}.\\  

\noindent{\it (5):}  Suppose first that $w_1$ and $w_2$ are distinct neighbors of $v$ in $I$.
By {\it(2)}, {\it(3)} and~{\it(4)}, they are nonadjacent, they have no neighbors in $U$, and neither is adjacent to $x$.   Furthermore, because $w_1$,~$w_2 \notin I_1$, both vertices are distinct from $x$.  Then $d(w_1) + d(u) \leq |I| - 3 + |U| = n-3$, a contradiction.
Hence $|N_I(v)| \leq 1$.

%By Claims~\ref{claim:bigI} and \ref{cl:bigI2}, $|I|-1 \geq |U| \geq |I_1|+1$.
Thus, since $|I_0| \geq 2$ by Claim~\ref{bigI0},  there exists $w \in I_0 \setminus N(v)$.
By applying the assumption $\sigma_2(G)\ge n-2$ to the nonadjacent vertices $v$ and $w$,
%and the fact that $w$ cannot be adjacent to both $v$ and $x$
we obtain
\[
|I|-2+d(v)\ge d(w)+d(v)\ge  n-2 = |U|+|I|-2.
\]
Hence, $|U| \leq d(v)$.  On the other hand, $d(v) \leq |U \setminus N(x)|+1 \leq |U|$, so $d(v)=|U|$.  It follows that $v$ is adjacent to all but one vertex of $U$, so that $I_1 = \{v,x\}$ and $x$ has exactly~one neighbor in $U$.  This completes the proof of Claim~\ref{cl:three}.  
\end{proof}

Let $v$ and $x$ be as in the statement of Claim~\ref{cl:three}.  By Claims \ref{claim:bigI} and~\ref{cl:three}{\it(5)}, $U$ is a clique, $x$ has exactly~one neighbor in $U$, and $v$ is adjacent to every other vertex of~$U$.

If $v$ and $x$ are not adjacent, let $z$ be the (only) neighbor of~$v$ in $I$.  By Claim~\ref{cl:three}{\it(5)}, $z \in I_0$, and by Claim~\ref{cl:three}{\it(4)}, $z$ is not adjacent to~$x$.  It follows from Observation~\ref{I0} that all other pairs of vertices in $I_0 \cup \{x\}$ are adjacent (cf.~Figure~\ref{fig:Gi}).  Thus, $G \in \G_2$.

If $v$ and $x$ are adjacent, then Claim~\ref{cl:three}{\it(5)} and Observation~\ref{I0} imply that $I_0 \cup \{x\}$ is a clique.  It follows that $G \in \G_3$.

In either case, we have a contradiction, the final one needed to complete the proof of Theorem~\ref{th:bestOre}.\end{proof}

\section{Proof of Theorem~\ref{thm:chvatalsharp}} %%%%%%%%%%%%%%%%%%%%%%%%%%%%%%%%%%%%%%%%%%%%%%%%

\begin{proof}[Proof of Theorem~\ref{thm:chvatalsharp}]
For $n \leq 12$, Theorem~\ref{thm:chvatalsharp} was verified using Nauty \cite{McKay2014}, so throughout the proof, we may assume that $n \ge 13$. 
The program we used is available at \OURURL.
Suppose then that $G$ is a graph of order $n$ that satisfies the weak Chv\'atal condition~\eqref{eqwch}.  Further, by way of contradiction, suppose that $m(G,2)> 2$ and that $G$ is connected and has at most one vertex of degree~$1$.  Also, let $V(G)=\{v_1,v_2, \ldots, v_n\}$, where $d(v_i)=d_i$ and $d_i\le d_j$ whenever $i\le j$, and let $L$ be the set of vertices of degree at least~$\frac{n-1}{2}$.
%and let $S=V\setminus L$. 

If $G$ satisfies the full Chv\'{a}tal condition~\eqref{eqch}, one of the following must hold:

\begin{itemize}
\item %Type 1:
$d(v)\ge \frac{n}{2}$ for all $v \in L$  and  $|L| > \frac{n}{2}$ or
\item %Type 2:
$d(v)>\frac{n}{2}$ for all $v \in L$ and $|L| \ge \frac{n}{2}$.
\end{itemize}
As one would expect, the weak Chv\'atal condition~\eqref{eqwch} results in slightly weaker conclusions.  However, \eqref{eqwch} still implies that $|L| \geq n/2$.
%For even $n$, $L$ is either Type 1 or
%\begin{itemize}
%\item Type $2^{-}$:  $d(v)\ge \frac{n}{2}$ for all $v \in L$  and  $|L| \geq \frac{n}{2}$.
%\end{itemize}
%
%For odd $n$, $L$ is either Type 2 or
%\begin{itemize}
%\item Type $1^{-}$:  $d(v)\ge \frac{n-1}{2}$ for all $v \in L$  and  $|L| \geq \frac{n+1}{2}$.
%\end{itemize}
Also, notice that if $n$ is even, then $d(v) \ge \frac{n}{2}$ for all $v \in L$, and if $n$ is odd, then $|L| \geq \frac{n+1}{2}$.

Let $I$ have maximum size among sets that can be activated by starting from two activated vertices and that satisfy $I\cap L\ne \emptyset$.
Furthermore, let $U=V\setminus I$, $L_I=L\cap I$ and $L_U=L\cap U$. 
Since we assumed for contradiction. $m(G,2) > 2$, we have $|U| > 0$. Notice that Observation~\ref{oneneighbor}, which says that each vertex in $U$ has at most~one neighbor in $I$, continues to hold.

\begin{claim}\label{claim:smallishI}
$|I| \le \frac{n+1}{2}$.
\end{claim}
\begin{proof}
Suppose, towards a contradiction, that $|I| > \frac{n+1}{2}$, so that $|U| < \frac{n-1}{2}$. Let $u \in U$ and note that since $u$ has at most one neighbor in $I$, $d(u) \leq |U|$, which implies that $d(v_{|U|}) \leq |U|$.  The weak Chv\'{a}tal condition~\eqref{eqwch} then implies that $d(v_{n-|U|}) \geq n-|U|-1$. 
Let $X = \{v_{n-|U|},\ldots,v_n\}$, so that $d(v) \geq n-|U|-1=|I|-1> |U|$ for all $v \in X$.  Note that $|X|=|U|+1$.

%If $|U| = |I| - 1$, then $n$ is odd, $|I| = \frac{n+1}{2}$ and $|U| = \frac{n-1}{2}$.  In this case \dots \textbf{????}

As $|U| \leq |I| - 2$, we have $n-|U|-1 > |U|$, which means that
%Up to relabeling the vertices, while maintaining the order of their degrees, we may assume that
$X\subseteq I$ since $d(u) \leq |U|$ for all $u \in U$.
Suppose that there are vertices $v$ and $w$ in $X$ with no neighbors in $U$, implying that they have degree equal to $|I|-1$ in $I$.
Because $G$ is connected, there exists $u \in U$ with a neighbor in $I$. Initially activating $\{v,u\}$ then activates $w$
in the second round and consequently activates $I \cup \{u\}$, which contradicts the maximality of~$|I|$.
Thus there is at most one vertex $v \in X$ with no neighbors in $U$.  Indeed, since $|X| = |U|+1$, $v$ is the unique member of $X$ that has no neighbors in $U$
and, by Observation~\ref{oneneighbor}, every $y \in X\setminus\{v\}$ has exactly one neighbor in $U$.  Hence, there is a perfect matching between $U$ and $X\setminus\{v\}$.

Let $w \in X \setminus\{v\}$ and let $u \in U$ be adjacent to $w$.  Initially activating $\{v,u\}$ activates $w$ in the first round, followed by all vertices in $I \cap N(w)$, since they are also adjacent to $v$. By the maximality of $|I|$ and the fact that $|N(w) \cap I| \geq |I|-2$,
there is exactly one vertex~$z \in I \setminus \{w\}$ that is not adjacent to $w$.
Moreover, $z$ must be adjacent only to $v$ in $I$, otherwise $I\cup\{u\}$ would be activated.  

If $z$ had a neighbor in  $U$,  then $z \in X$.
Hence, $2 \geq d(z) \geq |I| - 1$, which implies that $|I| \leq 3$ and thus $n \leq 5$, a contradiction.  Hence, $z$ is a vertex of degree one.  Also, because $z \notin X$, all of~$X$ becomes active.  If $u$ has a neighbor $u'$ in $U$, then $u'$ has a neighbor in $X$, and hence also becomes active since all of $X$ is activated.   This would contradict the maximality of $|I|$. Hence $u$ is also a vertex of degree one, a contradiction to the assumption that $G$ has at most one vertex of degree~one.  This concludes the proof of Claim~\ref{claim:smallishI}.
\end{proof}

\begin{claim}\label{claim:strictodd}
$|I| < \frac{n+1}{2}$.
\end{claim}
\begin{proof}
Assume otherwise.  By Claim~\ref{claim:smallishI}, $n$ is odd, $|I| = \frac{n+1}{2}$ and $|U| = \frac{n-1}{2}$. 

First we show that $G[I]$ does not contain two universal vertices. 
Suppose for a contradiction that $x$ and $y$ are two vertices of $I$, each with $|I|-1$ neighbors in $I$. 
By the connectedness of $G$, there exists an edge $zu$, where $z \in I$ and $u \in U$.
By symmetry, assume $x \neq z$.  
Initially activating $\{x,u\}$ activates $z$ in the first step. After the second step, $y$ is activated.
This activates $I \cup \{u\}$, contradicting the maximality of $|I|$.

%Observe that when $n$ is odd, the weak Chv\'atal condition implies there exists $L'$ such that

For $i = (n-3)/2$, the weak Chv\'{a}tal condition \eqref{eqwch} states that either
$d_i \geq i+1 = (n-3)/2+1 = (n-1)/2$ or 
$d_{n-i} \geq n-i-1 = n - (n-3)/2 - 1 =  (n+1)/2$.
This implies there exists $L'$ such that
%If (A) happens, then there are vertices |i,i+1,….,n| =   n - i + 1 =  n - (n-3)/2 + 1 = (n+5)/2 = (n-1)/2 + 3 vertices of degree  (n-1)/2.
%If (B) happens, then there are |n-i,…,n| = i+1 = (n-1)/1 vertices of degree  (n+1)/2
%I guess we should include this justification for (A) or (B) if it confused you. I hope I didn’t mess it up.

\begin{itemize}
\item[(A)] 
$d(v)\ge \frac{n-1}{2}$ for all $v \in L'$ and $|L'| \ge \frac{n-1}{2}+3$ or
\item[(B)] 
$d(v)\ge \frac{n+1}{2}$ for all $v \in L'$  and  $|L'| \ge \frac{n-1}{2}$.
\end{itemize}

If (B) holds,  Observation~\ref{oneneighbor} implies that $L' \subseteq I$. 
Moreover, every vertex in $L'$ must have at least one neighbor in $U$.
Hence $|L'| \leq |U|$ as every vertex in $U$ has at most one neighbor in $L' \subseteq I$. 
Therefore, $|U| = |L'| = \frac{n-1}{2}$.
Since $|U|=|L'|$, every vertex in $|L'|$ has exactly one neighbor in $U$ and thus every vertex in $L'$ is universal in $G[I]$.
Since $|L'| \geq 2$, this contradicts that $G[I]$ has at most~one universal vertex.

Now we assume that (A) holds. 
Since $|I| = \frac{n+1}{2}$ and $|L'| \geq  \frac{n+1}{2}+2$, there are at least two vertices, call them $u$ and $v$, in $L' \cap U$. 
Denote the neighbors of $u$ and $v$ in $I$ by $u_I$ and $v_I$, respectively; we first show that every vertex in $I$ has at most one neighbor in $U$.  Suppose otherwise, so that there exists $x \in I$ with at least two distinct neighbors $a$ and $b$ in $U$.  If $u_I \neq x$, then initially activating $\{u_I,v\}$ first activates $u$, then $U$ and $x$. Eventually, it activates at east $U \cup \{u_I,x\}$, contradicting the maximality of $|I|$.  Hence we may assume that $u_I = v_I = x$.   If $x$ has a single neighbor $y \in I$, then initially activating $\{y,v\}$ eventually activates $U \cup \{y,x\}$, again contradicting the maximality of $|I|$.   Consequently, we may then assume that $x$ has no neighbors in $I$;  this implies, however, that $x$ was in the initially activated set and therefore $|I|=2$ and thus $n\le 3$, a contradiction.

Consider, then, initially activating $\{u_I,v\}$, which activates $U \cup \{u_I\}$.
Notice that $|U \cup \{u_I\}| = |I|$ and $|U \cap L| \geq 2$. 
Hence, we can reuse the arguments from the beginning of the proof of Theorem~\ref{thm:chvatalsharp}, with $U \cup \{u_I\}$ in place of $I$. 
If we end in case (B) with $U \cup \{u_I\}$, we are done.
So, we may assume that (A) applies and then conclude that $U \cup \{u_I\}$ has the same properties as $I$.
%As $|U \cup \{u_I\}| = |I|$,  $U \cup \{u_I\}$ has the same properties as $I$.
In particular, this implies that $u_I$ has at most one neighbor in $I$.
Therefore, $d(u_I) = 2$, and by symmetry, $d(v_I) = 2$.

Since every vertex in $I$ has at most one neighbor in $U$, $\Delta(G) \leq \frac{n+1}{2}$.
%The weak Chv\'atal condition \eqref{eqwch} therefore implies that the initial part of the degree sequence of $G$ is termwise dominated by $2,3,4,\ldots$.  In particular, there is at most one vertex of degree at most 2 in $G$, contradicting the existence of $u_I$ and $v_I$.
The weak Chv\'atal condition \eqref{eqwch} therefore implies that the initial part of the degree sequence of $G$ termwise dominates the sequence~$2$, $3$, $4$,~$\ldots$.  In particular, there is at most~one vertex of degree at most~2 in $G$, contradicting the existence of $u_I$ and $v_I$.
\end{proof}

\begin{claim}\label{claim:strict}
$|I| < \frac{n}{2}$.
\end{claim}
\begin{proof}
Assume otherwise.  By Claim~\ref{claim:strictodd}, $|I| = |U|=\frac{n}{2}$, which implies that $n$ is even.  If $u \in L_U$, then $d(u)\geq \frac{n}{2}$. 
So, by Observation~\ref{oneneighbor}, $u$ is adjacent to all other vertices in $U$ and has exactly one neighbor in $I$.  If  $|L_U| \geq 2$, let $u$,~$x \in L_U$, let $v \in N_I(u)$ and initially activate $\{v,x\}$.  It follows that all of~$U$ is activated by the end of the second round. Because this activates $U\cup\{v\}$, it contradicts the maximality of $I$.  

Suppose, then, that $|L_U| \leq 1$, so that $|L_I| \geq \frac{n}{2}-1$.  As $n$ is even, every vertex in $L$ has degree at least $\frac{n}{2}$.
Hence every vertex in $L_I$ has at least one neighbor in $U$.
Since the number of edges between $I$ and $U$ is at most $\frac{n}{2}$, there is at most one vertex in $L_I$ with more than
one neighbor in $U$ and all the other vertices of $L_I$ are complete to $I$.
%Hence $G[I]$ is either a complete graph or a complete graph minus an edge, and all but at most one vertex in $I$ has a neighbor in $U$.  

Because
%$|I|=\frac{n}{2}\ge 4$,
$|L_I| - 1 \geq \frac{n}{2} - 2 \geq 2$,
there are vertices $v$ and $w$ in $I$ such that each vertex $u$ and $w$ is adjacent to all of~$I$ except for itself.  Let $u \in U$ be the unique neighbor of $v$ in $U$.  Activating $u$ and $w$ results in the activation of $I \cup \{u\}$,  contradicting the maximality of $I$.       
This concludes the proof of Claim~\ref{claim:strict}.
\end{proof}

Let
\begin{equation}\label{pdef}
p= \frac{n}{2}-|I|\ge \frac{1}{2}.
\end{equation}
Notice that $p$ is an integer if $n$ is even. 

 \begin{claim}\label{claim:bigLU}
 $|L_U|\ge 3$.
 \end{claim}
 \begin{proof}
If $p \geq 3$, the claim follows from $|L_I| \leq |I| \leq \frac{n}{2}-3$ and $|L_I|+|L_U| \geq \frac{n}{2}$.  So, we let $p < 3$
% We suppose first that $p>\frac{1}{2}$, and consider two cases based on the parity of $n$.
and assume for a contradiction that $|L_U| \leq 2$.
We distinguish the following three cases based on the parity of $n$ and the value of~$p$.\\ %being $\frac{1}{2}$.

\textbf{Case 1:} $n$ is even. %and $p>\frac{1}{2}$.

Since $n$ is even, \eqref{pdef} implies that $p \in \{1,2\}$.  Also, if $v\in L_I$, then $d(v) \geq n/2$. It follows from~\eqref{pdef} that $d_U(v) \geq \frac{n}{2} - (|I| - 1) = p+1$. By summing $d_U$ over all vertices in $L_I$ and recalling that every vertex in $U$ has at most one neighbor in $I$ we get $|U|\ge |L_I|(p+1)$. 
Since $|L_U| \leq 2$, we get $|L_I| \geq |L| - 2 \geq \frac{n}{2}-2$.
Therefore,
\[
\frac{n}{2}-2 \le |L_I|\le \frac{|U|}{p+1}=\frac{n/2+p}{p+1}.
\]
However, as $p \in \{1,2\}$ and thus $n \le 10$, this is a contradiction.\\

\textbf{Case 2:}  $n$ is odd and $p>\frac{1}{2}$.

Since $n$ is odd and $p > \frac12$, for every $v\in L_I$, we have $d_U(v)\ge p+\frac12$. Every vertex in $U$ has at most one neighbor in $I$, so $|U|\ge |L_I|(p+1/2)$. 
Since $|L| \geq \frac{n+1}{2}$, we obtain $|L_I| \geq \frac{n-3}{2}$. 
Therefore,
\[
\frac{n-3}{2} \le |L_I|\le \frac{|U|}{p+1/2}=\frac{n/2+p}{p+1/2}.
\]
However, as $p \in \{\frac{3}{2},\frac{5}{2}\}$ and $n > 9$, this inequality fails, a contradiction.\\

\textbf{Case 3:} $n$ is odd and $p = \frac{1}{2}$.

%Since $p = \frac12$, $n$ is odd.
%It remains to consider the case where $n$ is odd and $p = \frac12$.  
In this case, $|I| = \frac{n-1}{2}$ and, by assumption, $|L_I| \geq |L| - 2 \geq \frac{n-1}{2} - 1$.

Every vertex in $L_I$ has degree at least $\frac{n-1}{2}$,
and therefore must have a neighbor in $U$.  
Since every vertex in $U$ has at most one neighbor in $I$, there
are at most $\frac{n+1}{2}$ edges between $L_I$ and $U$.
Hence there are at most two vertices in $L_I$ with more than one neighbor in $U$.
If there are two such vertices in $L_I$, then both have exactly two neighbors in $U$.
If there is exactly one such vertex in $L_I$, then it has at most three neighbors in $U$.
%Hence there are at most either two vertices in $L_I$ that have two neighbors in $U$
%or one vertex in $L_I$ that has three neighbors in $U$. 

Consequently, $I$ is a complete graph of order at least $5$ except for either a single edge or two incident edges.
 As each vertex in $L_I$ has at least one neighbor in $U$, it is straightforward to select two vertices that, when initially activated, activate all of $I$ and at least~one vertex in $U$, a contradiction.  
 This concludes the proof of Claim~\ref{claim:bigLU}.
\end{proof}

%We will contradict the maximality of $I$ by initially activating two vertices in $L_U$.  The following claim will provide some useful structure.  

\begin{claim}\label{luclique}
$L_U$ is a clique.
\end{claim}

\begin{proof}
Assume otherwise, and let $u$ and $v$ be nonadjacent vertices in $L_U$. 
We claim that initially activating $u$ and $v$ generates a contradiction to the maximality of $I$.  

As every vertex in $U$ has at most one neighbor in $I$,  both $u$ and $v$ have at least 
$\frac{n-1}{2}-1$ neighbors among the other $\frac{n}{2}+p-2$ vertices of~$U$.
Let $r$ be the number of vertices in $U$ that are common neighbors of $u$~and~$v$.
By counting edges from $u$ and $v$ to the remainder of~$U$, we obtain
\[
2\left( \frac{n-1}{2}-1\right) \leq d_U(u)+d_U(v) \leq 2r + \bigl(|U \setminus\{u,v\} | - r\bigr) \leq 2r + \left(\frac{n}{2}+p-2 - r\right),
\]
which implies that
\[
\frac{n}{2} - p - 1 \leq r.
\]
Together with $\{u,v\}$, a total of $r+2\ge\frac{n}{2} - p +1 = |I| + 1$ vertices are activated by the second round, contradicting the maximality of $|I|$ and proving Claim~\ref{luclique}.
\end{proof}

\begin{claim}\label{cl:almostdone}
Vertices in $L_U$ have no neighbors in $I$.
\end{claim}
\begin{proof}
Suppose for a contradiction that $v \in I$ and $u \in L_U$ are adjacent. Let $w$ and $z$ be two vertices
in $L_U$ aside from $u$, which exist by Claim \ref{claim:bigLU}. Initially activate the set~$\{v,w\}$. In the first round, $u$ is activated, and in the second round, the remainder of $L_U$, including $z$, is activated. 

By counting edges from $u$, $w$, and $z$ to the remainder of $U$, and letting $r$ denote the number of vertices adjacent to at least two of $u$, $w$ or $z$, we get
\[
3\left( \frac{n-1}{2}-3\right) \leq d_U(u)-2+d_U(z)-2+d_U(w)-2 \leq 3r + \bigl(|U \setminus\{u,v,w\} | - r\bigr) \leq 3r + \left(\frac{n}{2}+p-3 - r\right),
\]
which implies that
\[
%n - p - 7.5 \leq 2r \\
\frac{n}{2} - \frac{p}{2} - \frac{15}{4} \leq r.
\]
When we include $u$, $v$, $w$, and~$z$, we see that there are at least $\frac{n}{2} - \frac{p}{2} + \frac{1}{4} > \frac{n}{2} - p = |I|$ active vertices, which contradicts the maximality of $|I|$
and concludes the proof of Claim~\ref{cl:almostdone}.
\end{proof}

To finish the proof, let $u$, $w$, $z \in L_U$ and initially activate the set $\{z,w\}$, so that $u$ is activated in the first round.
Now we count edges from $\{u,w,z\}$ to $U \setminus \{u,w,z\}$, again letting $r$ denote the number of vertices in $U \setminus \{u,w,z\}$ that are adjacent to at least two vertices in $\{u,w,z\}$.  By Claim~\ref{cl:almostdone},
\[
3\left( \frac{n-1}{2}-2\right) \leq d_U(u)-2+d_U(z)-2+d_U(w)-2 \leq 3r + \bigl(|U \setminus\{u,v,w\} | - r\bigr)\leq 3r + \left(\frac{n}{2}+p-3 - r\right).
\]
The leftmost and rightmost terms in the inequality above imply that
\[
%n - p - 4.5 &\leq 2r \\
\frac{n}{2} - \frac{p}{2} - \frac{9}{4} \leq r.
\]

%Together with $u$, $w$, and $z$, we get at least $\frac{n}{2} - \frac{p}{2} + \frac{1}{4} \geq \frac{n}{2} > |I|$
By including in the count $u$, $w$, and $z$ and by using \eqref{pdef}, we get that the number of activated vertices is at least $r+3 \geq \frac{n}{2} - \frac{p}{2} + \frac{3}{4} \geq \left(\frac{n}{2} - p\right) +\frac{p}{2} + \frac{3}{4}> |I|$, the final contradiction to the maximality of $|I|$. The existence of $I$ was a consequence of $m(G,2) > 2$. Hence $m(G,2) \leq 2$ holds and this finishes the proof of Theorem~\ref{thm:chvatalsharp}.
\end{proof}

\section*{Conclusion}

Theorem \ref{thm:chvatalsharp} gives a sharp degree condition that ensures that a graph $G$ satisfies $m(G,2)=2$, in that it provides a class of graphs that demonstrates the sharpness of the weak Chv\'{a}tal condition for this property.  However, Chv\'{a}tal-type conditions are often shown to be best possible in a different manner, which gives rise to a perhaps challenging open problem related to the work in this paper.  

Let $S=(d_1,\dots, d_n)$ and $S'=(d_1',\dots,d_n')$ be sequences of real numbers.  We say that $S$ \textit{majorizes} $S'$, and write $S\succeq S'$, if $d_i\ge d_i'$ for every $i$.  A sufficient degree condition ${\mathcal C}$ for a graph property ${\mathcal P}$ is \textit{monotone best possible} if whenever ${\mathcal C}$ does not imply that every realization of a degree sequence $\pi$ has property ${\mathcal P}$, there is some graphic sequence $\pi'\succeq \pi$ such $\pi'$ has at least one realization without property ${\mathcal P}$.  Note that it is possible that $\pi'=\pi$.  

The Chv\'{a}tal condition is a monotone best possible degree condition for hamiltonicity~\cite{Chv:hamiltonian}, and~\cite{BK} is a thorough survey of monotone best possible degree criteria for a number of graph properties.  However, it is easy to show that neither the Chv\'{a}tal condition~\eqref{eqch} or the weak Chv\'{a}tal condition~\eqref{eqwch} is monotone best possible for the property $m(G,2)= 2$.  

To see that the Chv\'{a}tal condition is not monotone best possible, consider the graphic sequence $$\pi=(i^i, (n-i-1)^{n-2i},(n-1)^i),$$ where $2\le i<\frac{n}{2}$ and the exponents represent the multiplicities of the terms of $\pi$.  The unique realization of $\pi$ is $K_i\vee(\overline{K_i}\cup K_{n-2i})$.  However, any sequence $\pi'$ such that $\pi'\succeq\pi$ must have at least two vertices of degree $n-1$, implying every realization $G$ of $\pi'$ has $m(G,2)=2.$ If $n$ is even, the sequence $$\pi=(i^i, (n-i-2)^{n-2i},(n-1)^i),$$ suffices to show that the 
weak Chv\'{a}tal condition is also not monotone best possible for the property $m(G,2)=2$.  

This gives rise to the following problem:

\begin{problem}
Determine a monotone best possible degree condition for the property ``$m(G,2)=2$''.  
\end{problem}

\textbf{Acknowledgement:}  The work in this paper was primarily completed during the 2016 Rocky Mountain - Great Plains Graduate Research Workshop in Combinatorics, held at the University of Wyoming in July 2016.  The authors would like to thank Gavin King and Tyrrell McAllister for many fruitful discussions during the workshop and the University of Wyoming for its hospitality.  
We would like to thank David Leach for pointing out an error in a previous version of Figure~\ref{fig:extremal} and anonymous referees for comments that improved the presentation of the paper.

\bibliographystyle{amsplain}
\bibliography{refs}

\end{document}